\documentclass{amsart}
\usepackage{graphicx} 
\usepackage{amsmath,amsthm,amsfonts}
\usepackage{amssymb} 
\newtheorem{theorem}{Theorem}
\newtheorem{lemma}[theorem]{Lemma}
\newtheorem{proposition}[theorem]{Proposition}
\newtheorem{corollary}[theorem]{Corollary}
\theoremstyle{definition}
\newtheorem{definition}{Definition}

\theoremstyle{remark}
\newtheorem*{remark}{Remark}
\usepackage{color}
\usepackage{hyperref}

\title{Graded identities of the first Weyl algebra and its generalizations}
\author[V. Futorny]{Vyacheslav Futorny}
\address{Shenzhen International Center for Mathematics, Southern University of Science and Technology, Shenzhen, Guangdong, P. R. China}
\email{vfutorny@gmail.com}
\thanks{V. Futorny was partially supported by  the NSF of China (12350710787 and 12350710178)}
\author[P. Koshlukov]{Plamen Koshlukov}
\address{651 Sergio Buarque de Holanda, Department of Mathematics, State University of Campinas, 13083-859 Campinas, SP, Brazil}
\email{plamen@unicamp.br}
\thanks{P. Koshlukov was partially supported by FAPESP Grant 2024/14914-9, and CNPq Grant 307184/2023-4}
\author[J. Schwarz]{Jo\~ao Fernando Schwarz}
\address{Shenzhen International Center for Mathematics, Southern University of Science and Technology, Shenzhen, Guangdong, P. R. China}
\email{jfschwarz.0791@gmail.com}
\date{}


\begin{document}

\begin{abstract}
We study the graded polynomial identities of the first Weyl algebra $W_1$ over an infinite field. The algebra $W_1$ satisfies no ordinary polynomial identities in characteristic 0. It admits a natural grading by the infinite cyclic group $\mathbb{Z}$. We construct a basis of the $\mathbb{Z}$-graded identities of $W_1$, which consists of a single identity. It expresses the fact that the degree 0 component in the grading is commutative. It is also well known that if  the characteristic of the base field is  $p>2$, then $W_1$ satisfies the same identities as the full matrix algebra of order $p$. In this situation, we describe the $\mathbb{Z}_p$-graded identities of $W_1$. Afterwards, using various combinatorial and algebraic tools we consider graded identities for various types of algebras generalizing the Weyl algebras. For example, we show that $\mathbb{Z}$-graded Galois rings in characteristic 0 satisfy the same graded identities as $W_1$ when they embed in a shift operator algebra $\mathcal{S}_1$, and as a consequence we obtain that these $\mathbb{Z}$-graded Galois rings are not PI. The same holds for the algebra of differential operators on $1$-dimensional torus. We obtain similar results for generalized Weyl algebras. We also deal with the graded identities for the quantum Weyl algebras and for the quantum plane. It turns out that in the latter case and when $q$ is the $\ell$-th primitive root of unity, one is led to study gradings by the group $\mathbb{Z}_\ell\times \mathbb{Z}_\ell$. In this case the quantum plane satisfies the same graded identities as the matrix algebra of order $\ell$. Finally we construct a natural $\mathbb{Z}$-grading on the universal enveloping algebra of $\mathfrak{sl}_2$, and prove that its $\mathbb{Z}$-graded identities are the same as those of $W_1$, in characteristic $0$.

\end{abstract}

\maketitle

\section{Introduction}
Let $K$ be an infinite field and Let $W_1$ be the first Weyl algebra over $K$. It is unital, and is generated by two elements $x$ and $y$ satisfying the relation $[y,x]=yx-xy=1$. It is well known that in characteristic 0 the algebra $W_1$ is a noncommutative noetherian domain and it is a central simple algebra. A well known realization of $W_1$ is the following. Consider the polynomial ring in one variable $K[t]$, it is an infinite dimensional vector space over $K$. Then $W_1$ is the algebra of linear transformations on $K[t]$ generated by $x$ which is the multiplication by $t$, $f(t)\mapsto tf(t)$, and $y$ is the derivation $f(t)\mapsto \displaystyle\frac{df} {dt}$. Observe that for such a realization one needs a field of characteristic 0, otherwise if the characteristic of the field is $p>0$ then $y^p\mapsto 0$. 

The Weyl algebras and their generalizations are of extreme importance in Algebra and also in Theoretical Physics. They bear resemblance with enveloping algebras of the Witt algebra. We note that the universal enveloping algebra of the first Witt algebra is a  domain but it is not Noetherian, though, see \cite{sierra}. We direct the readers to 

Denote $h=yx$, then the Weyl algebra is a free module over the polynomial algebra $K[h]$ with a basis consisting of 1 and $\{x^m, y^n\mid n\in\mathbb{N}\}$. It is well known and easy to check that if $f(h)\in K[h]$ then $x^mf(h) = f(h-m)x^m$ and $y^nf(h) = f(h+n)y^n$ for every $m$ and $n$. For $k\in\mathbb{Z}$ we put $z^0=1$, $z^k=x^k$ whenever $k>0$, and $z^k$=$y^{-k}$ if $k<0$, then we have $z^m z^n=g(h) z^{m+n}$ for some polynomial $g\in K[h]$. In fact it can be easily computed directly but we do not exhibit it as we will not need its specific form. 

Let $G$ be a group and $A$ an algebra (associative or whatever). If $A=\oplus_{g\in G} A_g$ where the $A_g$ are subspaces of the vector space $A$ such that $A_gA_h\subseteq A_{gh}$ for every $g$, $h\in G$ then $A$ is called a $G$-graded algebra. 
\begin{lemma}
\label{AisZgraded}
The algebra $A=W_1$ is $\mathbb{Z}$-graded where $A_n = K[h]z^n$ for every $n\in\mathbb{Z}$. 
\end{lemma}
\begin{proof}
Obvious by the above remarks. 
\end{proof}

The study of group gradings on algebras started long ago. Recall that the commutative polynomial ring is $\mathbb{Z}$-graded (by the degree), and this grading is widely used in Commutative Algebra. In the sixties, Wall classified the graded finite dimensional simple algebras assuming the grading group is $\mathbb{Z}_2$ and the field is algebraically closed, \cite{wall}. The complete classification of the group gradings on matrix algebras was obtained in \cite{bsz}. Gradings on algebras are of significant interest since they provide an additional structure on an algebra, which might make it ``easier" to study.

Let $X=\{x_1,x_2,\ldots\}$ be an infinite countable set, then the free associative algebra $K\langle X\rangle$ is the vector space with a basis consisting of 1 and all noncommutative monomials (words) in $X$. The multiplication is defined on the basis of elements by juxtaposition, and then extended by linearity. The elements of $K\langle X\rangle$ are called polynomials. The algebra $K\langle X\rangle$ has the following universal property. If $A$ is an associative $K$-algebra and $\varphi\colon X\to A$ is a map there exists a unique algebra homomorphism $\Phi\colon K\langle X\rangle\to A$ extending $\varphi$. 
\begin{definition}
A polynomial $f\in K\langle X\rangle$ is called a polynomial identity for the $K$-algebra $A$ whenever $f\in\ker\Phi$ for every homomorphism $\Phi\colon K\langle X\rangle\to A$. In other words, $f$ vanishes under every substitution of its variables by elements of $A$. 

The algebra $A$ is a PI algebra if there exists a nonzero $f\in K\langle X\rangle$ that is an identity for $A$. The set of all identities for $A$ forms an ideal $T(A)$ of $K\langle X\rangle$ which is called the T-ideal of $A$.
\end{definition}
Clearly $f=f(x_1,\ldots,x_n)$ is an identity for $A$ if and only if $f(a_1,\ldots, a_n)=0$ for every choice of $a_i\in A$. It is well known that $T(A)$ is closed under endomorphisms of $K\langle X\rangle$, that is why sometimes such ideals are called verbal ideals, in analogy with the case of groups. If $I$ is a T-ideal (that is $I$ is closed under endomorphisms) then $I=T(K\langle X\rangle/I)$, hence $I$ is the ideal of identities of an algebra.  

A monomial $m=x_1^{a_1}\cdots x_n^{a_n}\in K\langle X\rangle$ is of degree $\deg m=a_1+\cdots+a_n$. We can define a finer degree, the multidegree of $m$ which is the $n$-tuple $(a_1,\ldots, a_n)$. A polynomial $f$ is homogeneous if all of its monomials are of the same degree. It is multihomogeneous if $f$ is a linear combination of monomials of the same multidegree. Finally $f$ is multilinear whenever it is multihomogeneous of multidegree $(1,\ldots,1)$. In other words, $f(x_1,\ldots,x_n)$ is multilinear when it can be written as $f=\sum_{\sigma\in S_n} \alpha_\sigma x_{\sigma(1)}\cdots x_{\sigma(n)}$. Here $S_n$ stands for the symmetric group permuting $(1,\ldots,n)$, and $\alpha_\sigma\in K$. The following proposition is a basic fact in PI theory, see for example \cite[Proposition 4.2.3]{drensky}.
\begin{proposition}
\label{multihomoogeneous}
If $K$ is an infinite field then every T-ideal is generated, as a T-ideal, by its multihomogeneous elements. 

If, moreover, $K$ is of characteristic 0 then every T-ideal is generated by its multilinear elements. 
\end{proposition}

If the base field $K$ is of characteristic 0 then it is well known that the Weyl algebra $W_1$ does not satisfy polynomial identities: it suffices to observe that it is central simple and $\dim W_1=\infty$. On the other hand, if $K$ is an infinite field of characteristic $p>2$ then $W_1$ satisfies the same identities as the matrix algebra $M_p(K)$, see \cite{lopatin}. Recall that this fact can also be deduced from  \cite{belov} or \cite{tsuchimoto}. The research of the Weyl algebras in positive characteristic can be dated back to \cite{makar-limanov}, see also \cite{bavula_glasgow}. 

Graded polynomial identities have been object of extensive research since the fundamental work of Kemer \cite{kemer}. The $\mathbb{Z}_2$-graded identities were essential in Kemer's classification of the T-ideals in characteristic 0; this classification led Kemer to the positive solution of the well known problem proposed by W. Specht in the fifties: Is every T-ideal in characteristic 0 finitely generated? Since then there have been various results concerning graded identities. It turned out that the graded identities of an algebra may be easier to describe than the ordinary ones. Although if one knows the graded identities of a given algebra oner cannot say much about the ordinary ones, the former still provide plenty of information about the ordinary ones. Thus for example, the ordinary identities of $M_n(K)$ are known, even in characteristic 0, only for $n\le 2$ while the grade ones for several natural gradings have been obtained for every $n$. The matrix algebra $M_n(K)$ admits a grading by $\mathbb{Z}_n$ which is given as follows: the elementary matrices $e_{ij}$ are homogeneous in the grading, of degree $j-i$ modulo $n$. The graded identities for this grading were obtained by Vasilovsky \cite{vaszn} in characteristic 0, and by Azevedo \cite{azevedozn} in positive characteristic. Similarly, $M_n(K)$ is graded by $\mathbb{Z}$ assuming $e_{ij}$ is of degree $j-i$ for every $i$ and $j$. (In this grading, the homogeneous components of degree $\ge n$ and $\le -n$ are trivial.) The graded identities for such a grading are also known, see \cite{vasz, azevedoz}. 

Let us introduce the main objects in this direction. First we define the free graded algebra. Let $G$ be an arbitrary group, and let $X=\cup_{g\in G} X_g$ be a disjoint union of infinite countable sets. We form the free associative algebra $K\langle X\rangle$, and define a $G$-grading on it by setting all variables from $X_g$ to be of degree $g$. We write the elements of $X_g$ as $\{x_1^{g}, x_2^{g}, \ldots\}$. The degree $\deg_g m$ of the monomial $m=x_{i_1}^{g_1} \cdots x_{i_k}^{g_k}$ is the product $g_1\cdots g_k$. (If the group $G$ is additive then we substitute the product by the sum in the latter expression.) We denote the resulting $G$-graded algebra by $K\langle X\mid G\rangle$ when we want to stress that $G$ is the grading group. The definitions of a homogeneous, or graded, subspace, subalgebra, ideal, are the natural ones. If $A$ is a $G$-graded algebra then a polynomial $f\in K\langle X\mid G\rangle$ is a graded identity for $A$ if $f$ vanishes under every substitution of its variables by homogeneous elements of $A$, respecting the degrees. The notion of a graded T-ideal is defined as in the ordinary case. Recall that graded T-ideals are closed under graded endomorphisms. 

If $A$ is a $G$-graded algebra we denote by $T_G(A)$ the ideal of its graded identities in $K\langle X\mid G\rangle$. 

\begin{lemma}
\label{multihomgraded}
Let $K$ be an infinite field and let $A$ be a $G$-graded algebra with graded T-ideal $I=T_G(A)$. 

a) Then $I$ is generated, as a graded T-ideal, by its multihomogeneous elements.

b) If $K$ is of characteristic 0, then $I$ is generated by its multilinear elements. 
\end{lemma}

\begin{proof}
The proof follows word-by-word that of the ordinary case, see Proposition \ref{multihomoogeneous}. 
\end{proof}
We observe that a multilinear graded polynomial in the variables $x_1^{a_1}$, \dots, $x_n^{n}$ is of the form 
\[
\sum_{\sigma\in S_n}  \alpha_\sigma x_{\sigma(1)}^{a_{\sigma(1)}} \cdots x_{\sigma(n)}^{a_{\sigma(n)}}, \quad \alpha_\sigma\in K.
\]
In recent years there has been significant interest in studying the graded identities on infinite dimensional algebras. The results obtained so far are mainly concerned with Lie algebras graded by the group $\mathbb{Z}$, see for example \cite{cfpk1}, \cite{cfpk2}, and the references therein. 

In this paper we study the graded identities of the first Weyl algebra $A=W_1$. We prove that, in characteristic 0, the ideal of its graded identities is generated by a single graded identity. Later on we deal with the case of an infinite field of characteristic $p>0$, when the grading group is $\mathbb{Z}_p$. Finally, we extend these results to some important generalized Weyl algebras such as the quantum plane, the first quantum Weyl algebra, and $U(\mathfrak{sl}_2)$ using a $\mathbb{Z}$-graded version of the notion of a Galois ring (cf. \cite{futornyovsienko}).

\section{Graded identities in characteristic 0}\label{sec:graded-identities-weyl-char-0}
Throughout this section, we assume $K$ is a fixed field of characteristic 0, and we assume that the Weyl algebra $A=W_1$ is $\mathbb{Z}$-graded  as in Lemma~\ref{AisZgraded}. 

\begin{lemma}
\label{basicidentity}
The Weyl algebra $W_1$ satisfies the graded identity $[x_1^0, x_2^0]= x_1^0x_2^0-x_2^0x_1^0=0$. 
\end{lemma}

\begin{proof}
The proof is immediate: if $a_1$, $a_2\in A_0$ then $a_1=p_1(h)$, and $a_2=p_2(h)$, therefore $a_1$ and $a_2$ commute. 
\end{proof}

\begin{remark}
According to a well known theorem of Bergen and Cohen \cite{bergen} if $A$ is $G$-graded and if the neutral component in the grading satisfies an identity, that is there is a graded identity in variables of degree 0 only, then $A$ satisfies an ordinary identity whenever $G$ is a finite group. Our group is infinite hence the above lemma does not contradict Bergen and Cohen's theorem.
\end{remark}

We observe that by Lemma \ref{AisZgraded}, and the remarks preceding it, it follows that every element of $A_n$ can be written uniquely as $p(h)z^n$ for some polynomial $p\in K[h]$. 

\begin{lemma}
\label{canonicalform}
Let $M=x_1^{a_1} \cdots x_n^{a_n}$ be a graded monomial in $K\langle X\mid \mathbb{Z}\rangle$.  Substituting $x_i^{a_i}= p_i(h) z^{a_i} $ in $W_1$ we obtain \[
M= p_1(h)p_2(h-a_1) p_3(h-a_1-a_2) \cdots p_n(h-a_1-\cdots-a_{n-1}) z^{a_1+a_2+\cdots+a_n}.
\]
\end{lemma}
\begin{proof}
This was established before Lemma \ref{AisZgraded}: clearly one has $z^a p(h) = p(h-a) z^a$ for every $p\in K[h]$ and every $a\in\mathbb{Z}$. 
\end{proof}

We shall prove that the identity from Lemma~\ref{basicidentity} generates the ideal of graded identities of $W_1$ in characteristic 0. We consider multilinear polynomials only. We shall induct on the degree of our polynomials $n$. 

Clearly there exist no graded identities if $n=1$. Also, the Weyl algebra has no zero divisors therefore there are no graded monomials that are graded identities. 

\begin{lemma}
\label{degree2]}
If $f=\alpha x_1^{a_1}x_2^{a_2} - \beta x_2^{a_2} x_1^{a_1}$ is a graded identity for $W_1$, $\alpha$, $\beta\in K$ then either $a_1=a_2=0$ or $\alpha=\beta=0$.
\end{lemma}
\begin{proof}
Suppose that $(a_1,a_2)\ne (0,0)$. Put $x_1^{a_1} = p_1(h)z^{a_1}$, $x_2^{a_2}=p_2(h) z^{a_2}$. Then 
\[
f= (\alpha p_1(h)p_2(h-a_1) +\beta p_1(h-a_2)p_2(h))z^{a_1+a_2}. 
\]
Therefore we will have $\alpha p_1(h)p_2(h-a_1) +\beta p_1(h-a_2)p_2(h)=0$. If $a_1\ne 0$ we choose $p_2$ such that $p_2(h)=0$ and $p_2(-a_1)\ne 0$, and $p_1$ with the property that  $p_1(0)\ne 0$. In this way we obtain $\alpha=0$, then necessarily $\beta=0$. Similarly one treats the case when $a_1=0$ since in it one necessarily has $a_2\ne 0$.
\end{proof}

Since $[x_1^0,x_2^0]$ is a graded identity for $W_1$ we can, and shall work in the free graded algebra modulo the ideal of graded identities $I$ that is generated by the above graded polynomial. 

We recall some terminology introduced in \cite{GRWY}. 
\begin{definition}\label{def:balcom}
Let $u$ and $v$ be two graded monomials, then $v$ can be obtained from $u$ by means of a balanced commutation whenever $u=pqrs$ and $v=prqs$ where $p$ and $s$ are graded monomials (possibly empty), and $q$ and $r$ are graded monomials, both of $\mathbb{Z}$-degree 0. We shall denote this by $u\sim v$. 
\end{definition}

Clearly $\sim$ is an equivalence relation. 
Moreover one cannot distinguish $v$ and $u$, modulo the ideal $I$. In fact a stronger statement holds. 

Let $u=x_1^{a_1}\cdots x_n^{a_n}$ and $v=x_{\sigma(1)}^{a_{\sigma(1)}} \cdots x_{\sigma(n)}^{a_{\sigma(n)}}$ be two monomials where $v$ is obtained from $u$ by a permutation $\sigma$ of its variables. Substitute $x_i^{a_i}$ by $p_i(h) z^{a_i}$, $i=1$, \dots, $n$. Thus we obtain that in $W_1$  
\begin{align*}
u&=p_1(h)p_2(h-a_1) \cdots p_n(h-a_1-\cdots-a_{n-1})z^a,\\
v&=p_{\sigma(1)}(h) p_{\sigma(2)}(h-a_{\sigma(1)})  \cdots p_{\sigma(n)} (h-a_{\sigma(1)}-\cdots-a_{\sigma(n-1)}) z^a,
\end{align*}
where $a=a_1+\cdots+a_n$. Therefore $u=v$ in $W_1$ if and only if the two products of polynomials above are equal for every choice of the polynomials $p_i(h)$.  

\begin{proposition}
\label{balcom}
In the notation above, $u=v$ for every choice of the polynomials $p_i(h)\in K[h]$ if and only if $u\sim v$. 
\end{proposition}

\begin{proof}
If $u\sim v$ then obviously $u$ and $v$ are equal as elements of $W_1$. Now let $u$ and $v$ assume the same values under every substitution of the $x_i^{a_i}$ by elements of $W_1$, respecting the grading.  This means that the products of polynomials 
\begin{align*}
&p_1(h)p_2(h-a_1) \cdots p_n(h-a_1-\cdots-a_{n-1}),\\
&p_{\sigma(1)}(h) p_{\sigma(2)}(h-a_{\sigma(1)})  \cdots p_{\sigma(n)} (h-a_{\sigma(1)}-\cdots-a_{\sigma(n-1)})
\end{align*}
must be equal for the fixed $a_1$, \dots, $a_n$, $\sigma\in S_n$, and for every choice of the $p_i\in K[h]$. 

If $\sigma(1)=1$ then we use the fact that $W_1$ has no zero divisors, hence we cut out $x_1^{a_1}$ and we can proceed by induction on the length of the monomials. Hence suppose $\sigma(1)=k>1$. The ``shifts" $p_{\sigma(r)} (h-a_{\sigma(1)-\cdots-a_{\sigma(r-1)}})$ in the variable $h$ must be the same in both products, for every $r$. Therefore the shift $a_{\sigma(1)}+\cdots+a_{\sigma(k-1)}$ in $p_1$ in the second product must be equal to 0, that is the degree of the product $x_{\sigma(1)}^{a_{\sigma(1)}}\cdots x_{\sigma(k-1)}^{a_{\sigma1(k-1)}}$ must be 0.

Now in case $a_1=0$, we swap the positions of $x_{\sigma(1)}^{a_{\sigma(1)}}\cdots x_{\sigma(k-1)}^{a_{\sigma(k-1}}$ and $x_1^{a_1}$ in the second monomial, modulo $I$ this is harmless. Otherwise we look for a submonomial of the second monomial, starting from position $k$, and which is of degree 0. If there is such a submonomial then we swap its position with the position of the leading submonomial of length $k-1$ (which is also of degree 0). In case there is no such submonomial, this means we cannot obtain the same shifts as in the first monomial, and our proof is complete. 
\end{proof}

\begin{corollary}
\label{maincor}
Let $v$ be a multilinear monomial in the variables $x_1^{a_1}$, \dots, $x_n^{a_n}$. If we substitute $x_i^{a_i}$ by $p_i(h) z^{a_i}$ in $W_1$ and if we are given the polynomial 
\[
p_{\sigma(1)}(h) p_{\sigma(2)}(h-a_{\sigma(1)})  \cdots p_{\sigma(n)} (h-a_{\sigma(1)}-\cdots-a_{\sigma(n-1)})
\]
that appears as in the proof of Proposition \ref{balcom}, then we can recover the monomial, up to balanced commutations. 
\end{corollary}

\begin{proof}
The proof follows from Proposition \ref{balcom}. 
\end{proof}

\begin{theorem}
\label{mainthm}
Let $n\ge 2$ and let $f$ be a multilinear graded identity for $W_1$. Then $f\in I$, that is $f$ is a consequence of $[x_1^0,x_2^0]$.
\end{theorem}
\begin{proof}
Suppose, on the contrary, that $f$ is a graded identity for $W_1$ but $f\notin I$. Write $f$ as 
\[
\sum_{\sigma\in S_n}  \alpha_\sigma x_{\sigma(1)}^{a_{\sigma(1)}} \cdots x_{\sigma(n)}^{a_{\sigma(n)}}, \quad \alpha_\sigma\in K
\]
where at least one of the $\alpha_\sigma\ne 0$. In the above sum we can collect the monomials that can be obtained from each other by means of balanced commutations (it is obvious that $\sim$ is an equivalence relation). 

Therefore we write $f=\alpha_1 m_1+\cdots +\alpha_k m_k$. Here $\alpha_1$, \dots, $\alpha_k\in K$ are all nonzero, and if $i\ne j$ then $m_i\not\sim m_j$. We substitute $x_i^{a_i}$ for $p_i(h) z^{a_i}$ in $W_1$ and then represent each $m_r$ as $q_r(h) z^a$ where $a$ is the degree of $f$ in the $\mathbb{Z}$-grading. Thus $f=(q_1(h)+\cdots+ q_k(h)) z^a$. We know, by Proposition \ref{balcom} and Corollary \ref{maincor}, that the shifts of $p_1$, \dots, $p_n$ in the $q_r(h)$ are different. Therefore we can choose the polynomials $p_1$, \dots, $p_n$ in such a way that the coefficient $q_1(h)+\cdots+ q_k(h)$ is nonzero, and $f$ fails to be a graded identity for $W_1$.
\end{proof}

\begin{corollary}
\label{basischar0}
The ideal $T_{\mathbb{Z}}(W_1)$ of $\mathbb{Z}$-graded identities of $W_1$ is generated by the single identity $[x_1^0, x_2^0]$.
\end{corollary}

\section{Graded identities in characteristic $p>2$}
In this section we assume that $K$ is an infinite field of characteristic $p>2$, and we consider gradings and graded identities of $W_1$.

In this case $W_1$ is not a simple algebra; moreover it has a large centre. It is well known that $x^p$ and $y^p\in Z(W_1)$, the centre of $W_1$. It was proved in \cite{lopatin} that $W_1$ satisfies the same ordinary identities as the full matrix algebra $M_p(K)$. 

It is well known that the elements $x^ry^s$, $r$, $s\ge 0$, form a $K$-basis of the vector space of $W_1$ in any characteristic. Let $G=\mathbb{Z}_p$ be the cyclic group of order $p$, here we consider the following $G$-grading on $W_1$. We write $W_1=\oplus_{k=0}^{p-1} A_k$ where $A_k$ is the span of all $x^r y^s$, where $r-s\equiv k\pmod{p}$. 

\begin{lemma}
\label{p-grading}
The decomposition $W_1=\oplus_{k=0}^{p-1} A_k$ is  a $G$-grading on $W_1$.
\end{lemma}
\begin{proof}
It is easily shown by induction that 
\begin{equation}
\label{basic_mult}
x^ry^s\cdot x^uy^v = \sum_{i\ge 0} i! \binom{u}{i} \binom{s}{i} x^{r+u-i} y^{s+v-i}
\end{equation}
for every $r$, $s$, $u$, $v\ge 0$. We assume, as usual with binomial coefficients, that $\displaystyle\binom{t}{i} =0$ whenever $i>t$.

Since $r+u-i -(s+v-i) = r-s +u-v$, we have that $A_kA_l\subseteq A_{k+l\pmod{p}}$, and the above defines a $\mathbb{Z}_p$-grading on $W_1$.
\end{proof}

\begin{lemma}
\label{basicid1}
 $A_0$ is commutative, that is 
\begin{equation}
\label{deg0}
x^ry^s\cdot x^uy^v = x^uy^v \cdot x^ry^s
\end{equation}
whenever $r\equiv s\pmod{p}$ and $u\equiv v\pmod{p}$.
\end{lemma}

\begin{proof}
Compute, according to Eq.~\ref{basic_mult}, 
\begin{align*}
x^ry^s\cdot x^uy^v &= \sum_{i\ge 0} i! \binom{s}{i} \binom{u}{i} x^{r+u-i} y^{s+v-i}, \\
x^uy^v \cdot x^ry^s &=\sum_{i\ge 0} i! \binom{v}{i} \binom{r}{i} x^{u+r-i}y^{v+s-i}.
\end{align*}
Clearly if $i\ge p$ the corresponding terms with $i!$ vanish. 

Comparison of the coefficients of $x^{r+u-i} y^{s+v-i} = x^{u+r-i}y^{v+s-i}$ yields, assuming $r=s+ap$ and $u=v+bp$ for some integers $a$ and $b$, that
\[
i!\binom{s}{i} \binom{v+bp}{i} = s(s-1)\cdots (s-i+1)\frac{(v+bp)(v+bp-1)\cdots (v+bp-i+1)}{i!}
\]
and that
\[
i! \binom{v}{i} \binom{s+ap}{i} = v(v-1)\cdots (v-i+1) \frac{(s+ap)(s+ap-1)\cdots (s+ap-i+1)}{i!}.
\]
Since we can assume $i\le p-1$ then $i!$ is invertible and we cancel it out from the denominators. Then we observe that 
\begin{align*}
s(s-1)\cdots (s-i+1) &\equiv (s+ap)(s+ap-1)\cdots (s+ap-i+1)\pmod{p}, \\ 
v(v-1)\cdots (v-i+1) &\equiv (v+bp)(v+bp-1)\cdots (v+bp-i+1) \pmod{p}, 
\end{align*}
and we are done.
\end{proof}

\begin{lemma}
\label{basicid2}
Let $i-j\equiv \alpha\pmod{p}$, $r-s\equiv -\alpha\pmod{p}$, $u-v\equiv\alpha\pmod{p}$. Then 
\begin{equation}
\label{dega-aa}
x^iy^j\cdot x^r y^s\cdot x^uy^v = x^uy^v\cdot x^ry^s \cdot x^iy^j, 
\end{equation}
for every $\alpha\in \{0, 1,\ldots, p-1\}$.
\end{lemma}

\begin{proof}
Since $x^p$ and $y^p$ are central elements in $W_1$ we can assume that $i$, $j$, $r$, $s$, $u$, $v$ belong to $\{0,1,\ldots, p-1\}$. Suppose that $i\ge j$, $ r \geq s$, $u \geq v$ (the remaining cases are dealt with analogously). Then $x^i y^j= p_1(h) x^\alpha$, $x^r y^s= p_2(h)y^\alpha$, and $x^uy^v = p_3(h)x^\alpha$. Here $h=yx$, and $p_1$, $p_2$, $p_3$ are polynomials. Then we have 
\[
x^iy^j\cdot x^r y^s\cdot x^uy^v = p_1(h)p_2(h-\alpha)p_3(h) x^\alpha y^\alpha x^\alpha,
\]
and analogously 
\[
x^uy^v\cdot x^ry^s \cdot x^iy^j = p_1(h)p_2(h-\alpha)p_3(h) x^\alpha y^\alpha x^\alpha.
\]
This concludes the proof.
\end{proof}
We denote by $I$ the T-ideal of $\mathbb{Z}_p$-graded identities generated by the identities (\ref{deg0}) and (\ref{dega-aa}):
\begin{equation}
\label{basicidcharp}
z^0_1z_2^0 = z_2^0z_1^0, \qquad z_1^\alpha z_2^{-\alpha}z_3^\alpha = z_3^\alpha z_2^{-\alpha}z_1^\alpha, \quad \alpha\in\mathbb{Z}_p.
\end{equation}
Thus Lemmas~\ref{basicid1} and \ref{basicid2} yield that $I\subseteq T_{\mathbb{Z}_p}(W_1)$. Recall that $T_G(A)$ stands for the ideal of graded identities of the $G$-graded algebra $A$. We want to prove that $I=T_{\mathbb{Z}_p}(W_1)$. In order to achieve this we need several lemmas. For a monomial $m$ we denote by $\deg m$ its $\mathbb{Z}_p$-degree.

Since $K$ is an infinite field we consider multihomogeneous polynomials only. The following lemma is an example of a statement that is longer than its proof. 

\begin{lemma}
\label{generalz}
Let $m=m_1zm_2z\cdots m_kzm_{k+1}$ and $n=n_1zn_2z\cdots n_k zn_{k+1}$ be two monomials in the free $\mathbb{Z}_n$-graded algebra, of the same multidegree, and suppose further that neither of the monomials $m_1$, \dots, $m_{k+1}$, $n_1$, \dots, $n_{k+1}$ contains the variable $z$. Let $m=n$ for every substitution of the variables by elements of $W_1$. Then there exists a permutation $\theta$ of $\{1,2,\ldots, k\}$ such that $\deg(m_1zm_2z\cdots m_i) = \deg(n_{\theta(1)}zn_{\theta(2)}z\cdots n_{\theta(i)})$ for every $i=1$, \dots, $k$.
\end{lemma}
\begin{proof}
Applying (partially) Lemma~\ref{canonicalform}, we substitute $z$ by $p(h)z^a$ for some  $a \in \{ 0,1, \ldots , p-1 \}$. Then collecting the ``shifted" polynomials obtained from $p(h)$, and varying $p(h)$, we have the statement.
\end{proof}
Now we employ an idea from \cite{azevedozn}. 

\begin{lemma}
\label{case1}
Suppose that $m$ and $n$ are as in Lemma \ref{generalz}, and that $m=zM_1zM_2$ where $\deg(zM_1)=0$. Then $n\equiv zN\pmod{I}$ for some monomial $N$.
\end{lemma}

\begin{proof}
It follows from Lemma~\ref{generalz} that we can find submonomials in $n$ such that $n=n_1zn_2zn_3$, and moreover $\deg n_1=\deg(n_1zn_2)=0$. But then $\deg(zn_2)=0$ as well, and by identity~(\ref{basicid1}) we have $n_1\cdot zn_2\equiv zn_2\cdot n_1\pmod{I}$. Hence, modulo $I$ we can assume $n$ starts with $z$. 
\end{proof}

\begin{lemma}
\label{case2}
Let $m$ and $n$ be as in Lemma~\ref{generalz}, and suppose $m$ starts with the variable $z$. Suppose $m=m_1z_1z_2m_2$ and $n=n_1z_1n_2zn_3z_2n_4$ for some monomials $m_1$, $m_2$, $n_1$, $n_2$, $n_3$, $n_4$ (some of these may be empty). Here $z_1$ and $z_2$ are variables. Assume that $\deg m_1=\deg n_1$, and that $\deg(n_1z_1n_2)=0$. Then $n\equiv zN\pmod{I}$. 
\end{lemma}

\begin{proof}
First we note that there exist monomials $n_1$ and $n_2$ such that $\deg(n_1z_1n_2)=0$ because of Lemma~\ref{generalz}, and of $m$ starting with $z$. We can also suppose that $\deg(n_1z_1n_2zn_3) = \deg(m_1z_1)$ because these are the prefixes of $z_2$ in our monomials. We shall prove that $\deg(n_2zn_3)=0$. 

Indeed $\deg(n_1z_1n_2)=0$ implies $\deg n_2=-\deg(n_1z_1)=-\deg(m_1z_1)$. Then we observe that $\deg(zn_3) = \deg(mz_1)$ as these are the prefixes of $z_2$, and because $\deg(n_1z_1n_2)=0$. Therefore
\[
\deg(n_2zn_3) \equiv \deg n_2+\deg(zn_3) \equiv -\deg(m_1z_1) +\deg(zn_3) \equiv 0\pmod{p}.
\]
This proves our claim. But then $\deg(n_1z_1n_2) = \deg(n_2zn_3)=0$ yields 
\[
\deg(n_1z_1)\equiv -\deg n_2 \equiv \deg(zn_3)\pmod{p},
\]
and we apply the identity from Lemma~\ref{basicid2} for $n_1z_1$, $n_2$ and $zn_3$. Thus we obtain, modulo $I$, a monomial that starts with the variable $z$.
\end{proof}

\begin{proposition}
\label{case3}
Let $m$ and $n$ be two multihomogeneous monomials of the same multidegree in $K\langle X\mid\mathbb{Z}_n\rangle$, and suppose that $m=n$ under every evaluation on $W_1$. If $m$ starts with the variable $z$ then $n\equiv zN\pmod{I}$, for some monomial $N$. 
\end{proposition}
\begin{proof}
We write $m=z_1z_2\cdots z_t$ for some $t$, here $z_1=z$; there may be repetitions among the variables of $m$. If $1\le k<t$, then both $z_k$ and $z_{k+1}$ appear in the monomial $n$. Let $n=n_1zn_2$ where $\deg n_1=0$. If $n_1=1$ we are done, so assume $n_1\ne 1$. 

We write $n=N_1z_kN_2 = N_3z_{k+1}N_4$ for some monomials $N_i$, and suppose $\deg N_1=\deg(z_1\cdots z_{k-1})$, and $\deg N_3=\deg(z_1\cdots z_k)$. Now we compare the lengths of $n_1$ and $N_3$. We denote the length of a monomial $m$ by $|m|$.

If $|n_1| = |N_3|$ then $z_{k+1}=z_1=z$, and Lemma~\ref{case1} applies since $\deg n_1=\deg N_3 = 0$. Suppose that $|n_1|>|N_3|$, in this case the variable $z$ in $n$ is situated between $z_k$ and $z_{k+1}$. The remaining conditions of Lemma~\ref{case2} are satisfied and we apply that lemma. Therefore we assume, without loss of generality, that $|n_1|<|N_3|$. 

This means that there is some $s$, $1\le s\le t$ such that for every $k\ge s$ we have that the monomial $z_kz_{k+1}\cdots z_t$ appears in $n_1$ ``shuffled". It suffices to consider only the monomial $z_sz_{s+1}\cdots z_t$, that appears shuffled inside $n_1$. Moreover this shuffling of $z_kz_{k+1}\cdots z_t$ covers all of the monomial $n_1$. In this way we assume that $z_sz_{s+1}\cdots z_t$ and $n_1$ are multihomogeneous monomials and of the same multidegree. 

Let $z_q$ be the first variable in the monomial $n$ (and $n_1$). Then its corresponding $z_q$ cannot be to the left of $z_s$ in the monomial $m$ (otherwise we apply the above argument). Write $m$ as follows:
\[
m=(z_1\cdots z_{s-1}) (z_s \cdots z_{q-1})z_q (z_{q+1}\cdots z_t) = M_1 M_2 z_q M_3,
\]
where the $M_i$ stand for the corresponding monomials enclosed in the parentheses above. But this implies $\deg (M_1M_2)$ = 0 because in $n$, the variable $z_q$ is at the leftmost position. Additionally we have $\deg n_1=\deg (z_sz_{s+1}\cdots z_t) = \deg(M_2z_qM_3) = 0$ since $\deg n_1=0$. 

Now, exchange the roles of $m$ and $n$ in the above argument. Write then $n=n_1zn_2$, then we can decompose further $n_1= P_1P_2$ and $n_2=P_3P_4$ where the $P_i$ are monomials, such that $\deg (P_1P_2) = 0$, and $\deg (P_2zP_3)= 0$ as well. Hence we get $n=P_1P_2zP_3P_4$. 

We have two cases to consider. If $\deg P_2=0$ then $\deg (zP_3)=0$ and $\deg P_1=0$. In this case we apply the graded identity from Lemma~\ref{basicid1}, and put $z$ at the first position of $n$, modulo the identities of $I$. 

If, on the other hand, $\deg P_2\ne 0$, then $\deg P_2=-\deg P_1=-\deg(zP_3)$, and we apply the graded identity from Lemma~\ref{basicid2} to $P_1\cdot P_2\cdot zP_3$, and once again $z$ goes to the leftmost position. The proposition is proved.
\end{proof} 

\begin{theorem}
\label{mainthmcharp}
Let $K$ be an infinite field of characteristic $p>0$, then the ideal of $\mathbb{Z}_n$-graded identities of the Weyl algebra $W_1$ is generated by the identities from (\ref{basicidcharp}).
\end{theorem}
\begin{proof}
Since $K$ is infinite, $T_{\mathbb{Z}_n}(W_1)$ is generated by its multihomogeneous elements. We want to prove $T_{\mathbb{Z}_n}(W_1) = I$ where $I$ is the T-ideal of graded identities generated by the identities from (\ref{basicidcharp}). Suppose there exists $f(z_1,\ldots, z_n)$ which is an identity for $W_1$ but $f\notin I$, and take such an $f$ of smallest degree. Write $f=\sum \alpha_km_k$ where the $m_i$ are monomials and $\alpha_i\in K$ are scalars. Proposition \ref{case3} gives us that if $z_1$ appears in $f$ then for every $m_i$ we have $m_i\equiv z_1n_i\pmod{I}$ for some monomials $n_i$. Therefore $f=z_1\sum \alpha_in_i$ is a graded identity for $W_1$. But $W_1$ has no zero divisors, hence $\sum\alpha_i n_i$ is a graded identity for $W_1$. Its total degree is less than that of $f$, hence it lies in $I$. But then $f\in I$, a contradiction. 
\end{proof}
We recall here that the latter theorem states that in characteristic $p$, the Weyl algebra $W_1$ satisfies the same $\mathbb{Z}_p$-graded identities as the matrix algebra $M_p(K)$, see \cite{azevedozn} for the description of the graded identities of matrix algebras over infinite fields. At first glance this may seem surprising but it is well known that, at least when $K$ is algebraically closed, $W_1$ in prime characteristic is an Azumaya algebra over its centre. In fact in \cite{lopatin} the authors proved that in positive characteristic $p$, the algebras $M_p(K)$ and $W_1$ satisfy the same (ordinary) polynomial identities as long as $K$ is infinite. It should be observed that the above theorem from \cite{lopatin} can be obtained as a corollary from our theorem~\ref{mainthmcharp}. First we observe a well known and easy folkloric fact which implies immediately that $W_1$ and $M_p(K)$ satisfy the same identities over an infinite field of positive characteristic. 

\begin{lemma}
\label{folklore}
Let $A$ and $B$ be two $G$-graded algebras over an infinite field, both graded by the same group $G$. If $A$ and $B$ satisfy the same graded identities then they satisfy the same ordinary identities. 
\end{lemma}

Observe that the reverse does not hold. It suffices to take, for example, the Grassmann (or exterior) algebra $E$ of infinite dimension over an infinite field of characteristic different from 2. 
One can define the trivial $\mathbb{Z}_2$-grading on it, that is the 0-component equals $E$ while the 1-component is 0. On the other hand $E$ has a natural $\mathbb{Z}_2$-grading having various applications. Put $E=E_0\oplus E_1$ where $E_0$ is the centre of $E$ and $E_1$ is the ``anticommuting" part of $E$. It is easy to see that $[x,y]=0$ for every $x$, $y\in E_0$ but this is not an identity for the former grading. The 0-component in the former grading is the whole $E$, and $E$ is not commutative. Other examples of similar nature can be given just as easily.  

\section{Generalizations}
\subsection{$\mathbb{Z}$-graded Galois rings}

In this section we introduce a $\mathbb{Z}$-graded notion of Galois rings \cite{futornyovsienko}. All algebras and fields we consider are $K$-algebras.

Let $L$ be a field and $G$ a group, we denote by $L* G$ a skew group algebra of $G$. Recall that $L*G$ is the vector space $LG$ over $L$ with a basis consisting of the elements of $G$, that is the same vector space as that of the group algebra of $G$. The multiplication in it is ``almost" the usual one but with a twist that we define now. There is an action of $G$ on $L$, that is, a group homomorphism $\varphi\colon G\to Aut(L)$ such that $(a_1g_1) (a_2g_2) = (a_1\varphi(g_1)(a_2)) (g_1g_2)$. Formally, one should write $L*_\varphi G$ but we will omit the subscript $\varphi$ when the action is understood from the context. We observe that $L*G$ is naturally $G$-graded if we assume that the nonzero elements of $Lg$ are of degree $g$ for every $g\in G$.

The following definition is inspired by \cite[Definition 3]{futornyovsienko}. The authors of \cite{futornyovsienko} considered only the case of algebraically closed fields $K$ of zero characteristic. However, in \cite{schwarz} it was shown that with a small modification the definition works over any field. For the sake of simplicity we assume $\operatorname{char} K=0$.

\begin{definition}
    Let $\Gamma$ be an affine integral domain and $U$ a finitely generated associative $\Gamma$-algebra. Let $F$ be the field of fractions of $\Gamma$, $\theta$ a $K$-automorphism of $\Gamma$ of infinite order, extended canonically to a $K$-automorphism of $F$, with the identification $\langle \theta \rangle \simeq \mathbb{Z}$. Then $U$ is called a $\mathbb{Z}$\textit{-graded Galois} $\Gamma$-\textit{ring} in $F*\mathbb{Z}$, if $U$ is a subalgebra of $F*\mathbb{Z}$, $FU=UF=F*\mathbb{Z}$, and $U$ is $\mathbb{Z}$-graded. Moreover we require $\Gamma \subset U_0$, the neutral component in the grading on $U$, and the inclusion $U \hookrightarrow F*Z$ an injective homomorphism respecting the $\mathbb{Z}$-grading, where $F*\mathbb{Z}$ has the natural $\mathbb{Z}$-grading.
\end{definition}

\begin{definition}\label{def-S1}
    We denote by $\mathcal{S}_1$ the algebra $K(h)*\mathbb{Z}$, where the homomorphism $\varphi$ is defined on the free generator $\varepsilon$ of $\mathbb{Z}$ as $\varepsilon(h)=h-1$; we consider this action extended to the field of fractions $K(h)$ of $K[h]$.
\end{definition}

We return to the Weyl algebras $W_1$. Again, we set $h=yx$, and we consider the $\mathbb{Z}$-grading introduced in Lemma \ref{AisZgraded}.

\begin{proposition}\label{weyl-example}
$W_1$ is a $\mathbb{Z}$-graded Galois ring in $\mathcal{S}_1$.
\end{proposition}
\begin{proof}
We have a $K$-algebra homomorphism $\phi$ from $W_1$ to $\mathcal{S}_1$ that sends $x$ to $\varepsilon$ and $y$ to $h \varepsilon^{-1}$. Since $\phi$ is non-zero and $W_1$ is simple, $\phi$ is injective, and hence we can identify $W_1$ and $\phi(W_1)$ to get an embedding $\iota: W_1 \hookrightarrow \mathcal{S}_1$. Moreover, this embedding clearly preserves the $\mathbb{Z}$-grading. To finish, we need to show that $K(h)W_1=W_1K(h)=\mathcal{S}_1$. This follows from \cite[Propostion 3.7]{schwarz}.
\end{proof}

The following proposition is well known, at least in the case of ordinary (without any group grading) polynomial identities. In a sense, the proposition states that if $K$ is infinite, extending the field yields no ``new" identities. This means that we can work over an algebraically closed field, say the algebraic closure of $K$.

\begin{proposition}\label{stable-identities}
Let $A$ be an algebra over an infinite field $K$, and let $f$ be a $\mathbb{Z}$-graded identity for $A$. Then, given any commutative $K$-algebra $C$, $f$ is a $\mathbb{Z}$-graded identity for $A \otimes_K C$. Moreover, every graded identity for $A\otimes_K C$ whose coefficients are in $K$, is a graded identity for $A$.
\end{proposition}
\begin{proof}
    The proof of this fact is the same as the one for nongraded identities, using Lemma \ref{multihomgraded}, cf. \cite[Lemma 1.4.2]{gz}.
\end{proof}

\begin{theorem}\label{identities-Galois-rings}
    Let $U$ be a $\mathbb{Z}$-graded Galois ring in $\mathcal{A}=F*\mathbb{Z}$. Then the $T$-ideal of $\mathbb{Z}$-graded identities satisfies the property $T_{\mathbb{Z}}(U)=T_{\mathbb{Z}}(\mathcal{A})$.
\end{theorem}
\begin{proof}
    Recall that $\mathbb{Z}$-identities are preserved under scalar extensions. But $U$ is a $K$-subalgebra of $U \otimes_K F$, and the latter algebra is a scalar extension of $U$. Hence using Proposition \ref{stable-identities}, $T_{\mathbb{Z}}(U)= T_{\mathbb{Z}}(U \otimes_K F)$. By definition, $\mathcal{A}$ is a homomorphic image of $U \otimes_K F$, so $T_{\mathbb{Z}}(\mathcal{A})  \supset T_{\mathbb{Z}}(U)$. On the other hand $U$ embeds into $\mathcal{A}$, therefore $T_{\mathbb{Z}}(\mathcal{A}) \subset T_{\mathbb{Z}}(U)$. Hence we are done.
\end{proof}

\begin{remark}
    The same proof yields a strengthening of \cite[Theorem 6.6]{fhjs} for (ungraded) Galois rings $U$ in a fixed subring of a skew monoid ring $\mathcal{K}$: $T(U)=T(\mathcal{K})$.
\end{remark}
\begin{corollary}\label{identities-S1}
    The algebra $\mathcal{S}_1$ satisfies the same $\mathbb{Z}$-graded identities as $W_1$. In particular, the ideal of its graded identities is generated by the single identity $[x_1^0,x_2^0]=0$. 
 Moreover, the algebra $\mathcal{S}_1$ is not PI.
\end{corollary}
\begin{proof}
    The proof of the first part is combination of Corollary \ref{basischar0}, Proposition \ref{weyl-example}, and Theorem \ref{identities-Galois-rings}. The second statement follows from Lemma~\ref{folklore} and the fact that $W_1$ is not PI in characteristic 0.
\end{proof}

\begin{definition}\label{torus}
    We denote by $\mathcal{D}_1$ the algebra of differential operators on the torus $K^\times$.
\end{definition}

We clearly have $\mathcal{D}_1=\mathcal{D}(K[x,x^{-1}])=\mathcal{D}(K[x]_x)$, where $\mathcal{D}(.)$ denotes the ring of differential operators on a commutative $K$-algebra, and $K[x]_x$ is the localization of the polynomial algebra at the set $\{1, x, \ldots ,x^n, \ldots \}$. The rings of differential operators are compatible with localization \cite[Theorem 15.1.25]{mcconnell}, so $\mathcal{D}_1$ is the localization of $W_1$ at the set $\{1, x, \ldots ,x^n, \ldots \}$.

\begin{proposition}\label{example-torus}
    $\mathcal{D}_1$ is a $\mathbb{Z}$-graded Galois $K[h]$-ring in $\mathcal{S}_1$.
\end{proposition}
\begin{proof}
The $K$-algebra homomorphism $\phi: \mathcal{D}_1 \rightarrow \mathcal{S}_1$ that is given by $x \mapsto \varepsilon$, $x^{-1} \mapsto \varepsilon^{1}$ and $y \mapsto h \varepsilon^{-1}$, where $h=yx$ as before, induces an embedding $\mathcal{D}_1 \hookrightarrow \mathcal{S}_1$ like in the proof of Proposition \ref{weyl-example} (note that $\mathcal{D}_1$ is simple, being a localization of the simple ring $W_1$). We conclude that it is a $\mathbb{Z}$-graded Galois $K[h]$-ring by using once again \cite[Proposition 3.7]{schwarz}.
\end{proof}

\begin{corollary}\label{identities-torus}
    Let $\mathfrak{D}_1$ be the algebra of differential operators on the $1$-torus $K^\times$, with its natural $\mathbb{Z}$-grading. Then it has a basis for its $\mathbb{Z}$-graded polynomial identities $[x_1^0,x_2^0]$.
\end{corollary}
\begin{proof}
Combination of Corollary \ref{basischar0}, Proposition \ref{example-torus}, and Theorem \ref{identities-Galois-rings}.
\end{proof}

\begin{corollary}
In characteristic 0, the algebra $\mathcal{D}_1$ satisfies no ordinary polynomial identities.
\end{corollary}

\begin{proof}
Follows immediately from Lemma~\ref{folklore} and Corollary~\ref{identities-torus}. 
\end{proof}

\subsection{Generalized Weyl algebras}

Generalized Weyl algebras were introduced by V. V. Bavula in his PhD thesis, and first appeared in the paper \cite{bavula}. For a survey on these algebras and historical remarks about the development of such notions, see for example \cite{gaddis}. 

\begin{definition}\label{finite-GWA}
    Let $D$ be a ring, and $\sigma=(\sigma_1, \ldots, \sigma_n)$ an  $n$-tuple of commuting automorphisms: $\sigma_i \sigma_j = \sigma_j \sigma_i$, $i$, $j=1$, \dots, $n$. Let $a=(a_1,\ldots,a_n)$ be an $n$-tuple of non-zero elements belonging to the centre of $D$, such that $\sigma_i(a_j)=a_j$ whenever $j \neq i$. The \emph{generalized Weyl algebra} (GWA) $D(a, \sigma)$ of degree $n$ and base ring $D$ is generated over $D$ by $X_i^+$, $X_i^-$, $i=1$, \dots, $n$, satisfying the following relations
    \begin{subequations}\label{eq:GWA-relations}
    \begin{gather}
    X_i^+(d)= \sigma_i(d) X_i^+, \qquad  X_i^- d= \sigma_i^{-1}(d) X_i^-,\quad  d \in D,\\
    [X_i^+, X_j^+]=[X_i^-,X_j^-]=[X_i^+, X_j^-]=0, \quad i\neq j,\\
    X_i^- X_i^+ = a_i, \qquad X_i^+  X_i^- = \sigma_i(a_i).
    \end{gather}
    \end{subequations}
\end{definition}

The following proposition is important. Its proof can be found in  \cite[Proposition 1.3]{bavula}.

\begin{proposition}\label{GWA-noetherian-domain}
Let $D(a,\sigma)$ be a degree $n$ GWA with base ring $D$. If $D$ is left or right Noetherian, so is $D(a,\sigma)$. If $D$ is a domain, so is $D(a,\sigma)$.  
\end{proposition}

If the degree of the GWA $D(a,\sigma)$ is 1, we write $x=X_1^+$ and $y=X_1^-$. Then it has a $\mathbb{Z}$-grading defined as follows: $D(a,\sigma) =\bigoplus_{n \in \mathbb{Z}} D z^n$, where $z^n=x^n$ if $n \geq 0$ and $z^n=y^{-n}$ if $n <0$.

\begin{proposition}\label{GWA-Galois-rings}
Let $A=D(a,\sigma)$ be a degree $1$ GWA with a base ring $D$ that is an affine commutative domain and $\sigma$ an isomorphism of infinite order. Then $A$ is $\mathbb{Z}$-graded Galois $D$-ring in $F*\mathbb{Z}$, where $F = \operatorname{Frac} D$ and $\varepsilon \in \mathbb{Z}$ acts on $F$ by the natural extension of $\sigma$ from $D$ to its field of fractions $F$.
\end{proposition}
\begin{proof}
    The map from $A$ to $F*\mathbb{Z}$ that sends $D$ to $D$, $x$ to $\varepsilon$ and $y$ to $a \varepsilon^{-1}$ is an embedding of $A$ in $F*Z$, which preserves the $\mathbb{Z}$-grading of the algebras and realizes $A$ as a $\mathbb{Z}$-graded Galois $D$-order in $F*Z$ by \cite[Proposition 7.6]{schwarz}.
\end{proof}

We put $D=K[h]$, the polynomial ring in one variable. 
Then the algebra $D(a,\sigma)=K[h](a,\sigma)$ is called \textit{classical} whenever $\sigma(h)=h-1$, and \textit{quantum} if $\sigma(h)=qh$, for some $q \in K^\times$. In most of the applications one usually assumes as well that $q \neq \pm 1$. It is well known that, up to an isomorphism, every degree 1 GWA of the form $K[h](a,\sigma)$ that is not commutative is either classical or quantum, see \cite[Proposition 2.1.1]{richardsolotar}. 

\begin{remark}
    Sometimes one allows the base ring to be $K[h^\pm]$ for generalized Weyl algebras of quantum type. We observe that in this case the results are similar to those of the classical type.
\end{remark}

Defining, as in Section \ref{sec:graded-identities-weyl-char-0}, $z^0=1$, $z^n=x^n$ if $n \ge 0$, and $z^n=y^{-n}$ if $n <0$, we have that $A=K[h](a,\sigma)$ is $\mathbb{Z}$-graded, $A=\bigoplus_{k \in \mathbb{Z}} A_k$, where $A_k=K[h] z^k$.

\begin{corollary}\label{identities-classical-gwa}
    If $A$ is a degree $1$ GWA of classical type, the basis for its ideal of $\mathbb{Z}$-graded polynomial identities is $[x_1^0,x_2^0]=0$. Furthermore $A$ is not PI. 
\end{corollary}
\begin{proof}
$A$ is a $\mathbb{Z}$-graded Galois $K[h]$-ring in $\mathcal{S}_1$ by Proposition \ref{GWA-Galois-rings}. The result follows from Theorem \ref{identities-Galois-rings} and Corollary \ref{identities-S1}.
\end{proof}

\subsection{Quantized Weyl algebra and the quantum plane at non roots of unity}

In this subsection we assume that $\operatorname{char} K =0$. However, we point out that the definition of the quantum plane and its realization as a GWA do not require any assumption on the base field, as shown in the proof of Proposition \ref{quantum-plane}.

Let $q \in K$ be a parameter different from $0$ or $1$. The \textit{quantum Weyl algebra} $W_1(q)$ is the quotient of the free algebra $K \langle x, y \rangle$ by the relation $yx - q xy =1$.

\begin{proposition}
The quantum Weyl algebra $W_1(q)$ can be realized as a quantum generalized Weyl algebra as follows: $W_1(q) \simeq K[h](a,\sigma)$ where $a=h-1$ and $\sigma(h)=qh$.
\end{proposition}
\begin{proof}
This is the contents of \cite[Example 2.2]{gho}.
\end{proof}

If $q \in K$ is as before, the \textit{quantum plane} is the quotient of the free algebra $K \langle x, y \rangle$ by the relation $yx - q xy =0$. It is usually denoted as $K_q[x,y]$. 

\begin{proposition}\label{quantum-plane}
The quantum plane $K_q[x,y]$ can be realized as a quantum generalized Weyl algebra of the form $K[h](a,\sigma)$ with $a=h$ and $\sigma(h)=qh$. 
\end{proposition}
\begin{proof}
    There is a map $\phi: K \langle X^+, X^-, h \rangle \rightarrow K_q[x,y]$ that sends $X^-$ to $x$, $X^+$ to $y$, and $h$ to $xy$. Therefore $\phi (X^+ h) = y xy=q xyy=\phi(\sigma(h)X^+)$; similarly $\phi (X^- h)= \phi(\sigma^{-1}(h) X^-)$. Finally, it is clear that $\phi(X^- X^+)=\phi(h)$ and $\phi(X^+ X^-)=\phi(\sigma(h))$. Hence we have an induced map $\widehat{\phi}:K[h](a,\sigma) \rightarrow K_q[x,y]$. Since $x$, $y$ are in the image of this map, it is surjective. We note that $\operatorname{GKdim} K_q[x,y]=2$ \cite[Lemma II.7.9]{brown}, and that $\operatorname{GKdim} K[h](h,\sigma) =2$ as well \cite[Lemma 3.2]{zmz}. By Proposition \ref{GWA-noetherian-domain}, $K[h](h,\sigma)$ is a domain. We show that $\widehat{\phi}$ is injective, and hence an isomorphism. Otherwise, its kernel would be a non-zero two sided ideal $I$, which is essential (i.e., intersects every other non-zero one sided ideal non-trivially) since $K[h](h,\sigma)$ is a domain, such that $K[h](h,\sigma)/I \simeq K_q[x,y] $. But the right hand side has Gelfand-Kirillov 2, and the left hand one, by \cite[8.3.2(v), 8.3.6(i)]{mcconnell} has Gelfand-Kirillov dimension $1$, as $I$ is an essential ideal. This is absurd. Hence $\widehat{\phi}$ is an isomorphism.
\end{proof}

In more concrete terms, set $A=K_q[x,y]$. Then $A=\bigoplus_{n \in \mathbb{Z}} A_n$, where $A_0=K[xy]$, $A_n = K[xy]y^n$ if $n \geq 0$, $A_n = K[xy]x^{-n}$ if $n <0$. A little bit different of the Weyl algebra case, in view of Proposition \ref{quantum-plane}, we set, for an indeterminate $t$, $t^n=y^n$, if $n \geq 0$, and $t^n=x^{-n}$ if $n<0$.

\begin{remark}
Proposition \ref{quantum-plane} and hence the grading described in the above paragraph take place also if $q$ is a root of unity.    
\end{remark}

We will now determine the ideal of $\mathbb{Z}$-graded identities of the quantum plane $K_q[x,y]$ when $\operatorname{char} K=0$ and $q$ is not a root of unity. In this situation, every identity follows from  multilinear ones, according to Lemma \ref{multihomgraded}.

\begin{lemma}\label{identities-degree-0-quantum-plane}
The quantum plane $K_q[x,y]$ satisfies the graded identity $[x_1^0,x_2^0]=0$.
\end{lemma}
\begin{proof}
    This is clear, since the degree $0$ component of the quantum plane is the commutative algebra $K[h]$.
\end{proof}

Now we follow almost literally the arguments given in Proposition \ref{balcom}, Corollary \ref{maincor}, Theorem \ref{mainthm} and Corollary \ref{basischar0} above.

For two monomials $u$, $v$ from $K\langle X\mid\mathbb{Z} \rangle$ we write $u \sim v$ in the same sense as in Definition \ref{def:balcom}. Recall that $u\sim v$ means that $v$ can be obtained from $u$ by means of a sequence of balanced commutations. In other words we can swap the positions of two consecutive submonomials in case both of them are of homogeneous degree 0. Let $I$ be the $T_\mathbb{Z}$-ideal generated by the identity $[x_1^0, x_2^0]=0$.

Let $u$, $v$ be two monomials in $K\langle X \mid \mathbb{Z} \rangle$. We write $u \equiv v$ if $u$ and $v$ always coincide when evaluated on $K_q[x,y]$, respecting the $\mathbb{Z}$-gradings on both algebras.

\begin{proposition}
\label{balcom2}
In the notation above, assume $q$ is not a root of unity. Then $u\equiv v$ for every choice of the polynomials $p_i(h)\in K[h]$ if and only if $u\sim v$. 
\end{proposition}

\begin{proof}
If $u\sim v$ then by the graded identity $x_1^0,x_2^0]=0$ that holds in $K_q[x,y]$, it follows that $u$ and $v$ are equal as elements of $K_q[x,y]$. Now let $u$ and $v$ assume the same values under every substitution of the $x_i^{a_i}$ by elements of $K_q[x,y]$, respecting the grading. We substitute $x_i^{a_i}$ by $p_i(h) t^{a_i}$, $i=1$, \dots, $n$, where $p_i(h) \in K[h]$. We obtain that, in $K_q[x,y]$   the products of polynomials 
\begin{align}
\label{eq-quantum-1}    
u&=p_1(h)p_2(q^{a_1}h) \cdots p_n(q^{a_1 + \cdots +a_{n-1}}h)t^a,\\
\label{eq-quantum-2}
v&=p_{\tau(1)}(h)p_{\tau(2)}(q^{a_{\tau(1)}}h) \cdots p_{\tau(n)}(q^{a_{\tau(1)} + \cdots + a_{\tau(n-1)}}h)t^a,
\end{align}
must be equal for the fixed $a_1$, \dots, $a_n$, $\tau\in S_n$, and for every choice of the polynomials $p_i\in K[h]$. 

If $\sigma(1)=1$ then we use the fact that $K_q[x,y]$ has no zero divisors, hence we cut out $x_1^{a_1}$ and we can proceed by induction on the length of the monomials. Hence suppose $\sigma(1)=k>1$. The ``dilatations" $p_{\sigma(r)} (q^{a_{\tau(1)} + \ldots a_{\tau(r)}} h)$ in the variable $h$ must be the same in both products, for every $r$. Therefore the dilatation $q^{a_{\tau(1)}+\cdots+a_{\tau(k-1)}}$ in $p_1$ in the second product must be equal to $1$, that is the degree of the product $x_{\sigma(1)}^{a_{\sigma(1)}}\cdots x_{\sigma(k-1)}^{a_{\sigma1(k-1)}}$ must be 0.

Now in case $a_1=0$, we swap the positions of $x_{\sigma(1)}^{a_{\sigma(1)}}\cdots x_{\sigma(k-1)}^{a_{\sigma(k-1}}$ and $x_1^{a_1}$ in the second monomial, modulo $I$ this is harmless. Otherwise we look for a submonomial of the second monomial, starting from position $k$, and which is of degree 0. If there is such a submonomial then we swap its position with the position of the leading submonomial of length $k-1$ (which is also of degree 0). In case there is no such submonomial, this means we cannot obtain the same shifts as in the first monomial, and our proof is complete. 
\end{proof}

\begin{corollary}
\label{maincor2}
Let $v$ be a multilinear monomial in the variables $x_1^{a_1}$, \dots, $x_n^{a_n}$. If we substitute $x_i^{a_i}$ by $p_i(h) t^{a_i}$ in $K_q[x,y]$ and if we are given the polynomial 
\[
p_{\tau(1)}(h)p_{\tau(2)}(q^{a_{\tau(1)}}h) \cdots p_{\tau(n)}(q^{a_{\tau(1)} + \cdots + a_{\tau(n-1)}}h)
\]
that appears as in the proof of Proposition \ref{balcom2}, then we can recover the monomial, up to balanced commutations. 
\end{corollary}

\begin{proof}
The proof follows from Proposition \ref{balcom2}. 
\end{proof}

\begin{theorem}
\label{mainthm2}
Let $n\ge 2$ and let $f$ be a multilinear graded identity for $K_q[x,y]$. If $q$ is not a root of unity then $f\in I$, that is $f$ is a consequence of $[x_1^0,x_2^0]$.
\end{theorem}
\begin{proof}
Suppose, on the contrary, that $f$ is a graded identity for $K_q[x,y]$ but $f\notin I$. Write $f$ as 
\[
\sum_{\sigma\in S_n}  \alpha_\sigma x_{\sigma(1)}^{a_{\sigma(1)}} \cdots x_{\sigma(n)}^{a_{\sigma(n)}}, \quad \alpha_\sigma\in K
\]
where at least one of the $\alpha_\sigma\ne 0$. In the above sum we can collect the monomials that can be obtained from each other by means of balanced commutations. Recall that it is immediate to check  that $\sim$ is an equivalence relation. 

Therefore we write $f=\alpha_1 m_1+\cdots +\alpha_k m_k$. Here $\alpha_1$, \dots, $\alpha_k\in K$ are all nonzero, and if $i\ne j$ then $m_i\not\sim m_j$. We substitute $x_i^{a_i}$ for $p_i(h) t^{a_i}$ in $K_q[x,y]$ and then represent each $m_r$ as $q_r(h) t^a$ where $a$ is the degree of $f$ in the $\mathbb{Z}$-grading. Thus $f=(q_1(h)+\cdots+ q_k(h)) t^a$. We know, by Proposition \ref{balcom2} and Corollary \ref{maincor2}, that the dilatations of $p_1$, \dots, $p_n$ in the $q_r(h)$ are different. Therefore we can choose the polynomials $p_1$, \dots, $p_n$ in such a way that the coefficient $q_1(h)+\cdots+ q_k(h)$ is nonzero, and $f$ fails to be a graded identity for $W_1$.
\end{proof}

\begin{corollary}
\label{basischar02}
The ideal $T_{\mathbb{Z}}(K_1[x,y])$ of $\mathbb{Z}$-graded identities of $K_q[x,y]$ is generated by the single identity $[x_1^0, x_2^0]$.
\end{corollary}

\begin{proposition}
    Assume that $\operatorname{char} K=0$. Then $W_1(q)$ and $K_q[x,y]$ are (ungraded) PI-rings if and only if $q$ is an $\ell$-th root of unity. In this case, their PI-degree is $\ell$.
\end{proposition}
\begin{proof}
This follows from \cite[Corollary I.14.1, Proposition I.14.2]{brown}. 
\end{proof}

\begin{definition}\label{def-s1q}
 We denote by $\mathcal{S}_1(q)$ the algebra $K(h)*\mathbb{Z}$, where the homomorphism $\varphi$ is defined on the free generator $\varepsilon$ of $\mathbb{Z}$ as $\varepsilon(h)=qh$; we consider this action extended to $K(h)$.
\end{definition}

\begin{proposition}
\label{gwaqt}
   If $A$ is a degree 1 GWA of quantum type and $q$ is not a root of unity, then $A$ is a $\mathbb{Z}$-graded Galois $K[h]$-ring in $\mathcal{S}_1(q)$
\end{proposition}

\begin{proof}
The proof is an immediate consequence from Proposition \ref{GWA-Galois-rings}.
\end{proof}

\begin{corollary}
    If $A$ is a  GWA of quantum type of degree $1$ and $q$ is not a root of unity, then one has for the corresponding ideals of $\mathbb{Z}$-graded identities: $$T_{\mathbb{Z}}(A)=T_{\mathbb{Z}}(\mathcal{S}_1(q))=T_{\mathbb{Z}}(K_q[x,y]).$$
\end{corollary}

\begin{proof}
The proof follows from Proposition~\ref{gwaqt}. 
\end{proof}

\subsection{Quantum plane at roots of unity}

In this subsection we assume once again that $\operatorname{char} K=0$. We also assume that $K$ is algebraically closed.

\begin{proposition}
\label{iso_qplane}
    Let $K_q[x,y]$ be the quantum plane where $q$ is an $\ell$-th root of unity. Let $I \lhd K_q[x,y]$ be the two-sided ideal generated by $x^\ell$ and $y^\ell$. Then $K_q[x,y]/I$ is isomorphic to $M_\ell(K)$
\end{proposition}
\begin{proof}
        As in \cite[Example 3.1.7]{gz} we have matrices $A$, $B \in M_\ell(K)$ with $AB=qBA$, $A^\ell=B^\ell=\operatorname{Id}$. Hence we have a non-zero homomorphism $\phi\colon K_q[x,y] \to M_\ell(K)$ such that $x \mapsto B$ and $y \mapsto A$. Clearly $I \subset \operatorname{ker} \phi$, so we have an induced map $\widehat{\phi}\colon K_q[x,y]/I \to M_\ell(K)$. Then $\widehat{\phi}$ is clearly surjective; since $ K_q[x,y]/I$ and $M_\ell(K)$ both have the same dimension $\ell^2$ as $K$-vector spaces, $\widehat{\phi}$ is an isomorphism.
\end{proof}

We can consider $K_q[x,y]$, $q$ an $\ell$-th root of unity, as a $\mathbb{Z}_\ell\times \mathbb{Z}_\ell$-graded algebra in an obvious manner. Since $x^\ell$ and $y^\ell$ are clearly central elements, one may view $K_q[x,y]$ as a module over its centre $K[x^\ell,y^\ell]$, and then attribute degree $(u,v)\in \mathbb{Z}_\ell\times \mathbb{Z}_\ell$ to the monomial $x^ry^s$ whenever $r\equiv u\pmod{\ell}$ and $s\equiv v\pmod{\ell}$. One can ask what the $\mathbb{Z}_\ell\times \mathbb{Z}_\ell$-graded identities for this grading are. Clearly they are the same as those for the Pauli grading on $M_\ell(K)$. In it, the matrices $A$ and $B$ in the proof of Proposition \ref{iso_qplane} are, respectively, $A=diag(q^{\ell-1}, q^{\ell-2}, \ldots, q, 1)$ and $B= e_{12}+e_{23}+\cdots+ e_{\ell-1,\ell}+ e_{\ell,1}$. We resume this as follows.

\begin{corollary}
Suppose $q$ is an $\ell$-th root of unity. Then if $G=\mathbb{Z}_\ell\times \mathbb{Z}_\ell$ one has $T_G(K_q[x,y])=T_G(M_\ell(K))$. Then the ideal of $G$-graded identities for $K_q[x,y]$ is generated by the graded identities 
\[
x_\mu x_\nu = \beta(\mu,\nu) x_\nu x_\mu.
\]
Here $\mu=(r,s)$, $\nu=(t,w)\in G$, and $\beta(\mu,\nu) = q^{ts-rw}$.
\end{corollary}

\begin{proof}
Direct and easy computation shows that $x^ry^s\cdot x^ty^w = q^{ts-rw} x^ty^w\cdot x^ry^s$. The remaining part of the proof follows from \cite[Theorem 3.4]{baht_dr}.  
\end{proof}
We observe that  \cite[Theorem 3.4]{baht_dr} proves a general result about gradings on matrix algebras by considering the above $\beta$ as an arbitrary skew-symmetric bicharacter. Further developments in this direction can be found in \cite{ckp}. 

Observe that the above corollary shows that, modulo its centre, $K_q[x,y]$ behaves similarly to the $\ell\times\ell$ matrices in case $q$ is a primitive $\ell$-th root of unity. One might ask then what the remaining group gradings on $K_q[x,y]$ are and what the corresponding graded identities are. The group gradings can be described by mimicking the approach from \cite{bsz}, considering the elementary and fine gradings, and then combining them. As for the graded identities the situation becomes "messier" since, in addition to the "usual" graded identities there may appear monomial ones. The latter are more elusive, and we leave the discussion for another occasion. 

\subsection{$U(\mathfrak{sl}_2)$ as a $\mathbb{Z}$-graded algebra}

\begin{proposition}\label{sl2}\cite[Example 1.2(3)]{bavula}
    Consider the Lie algebra $\mathfrak{sl}_2$ over a field of characteristic $0$ with its canonical basis elements $e$, $f$, $h$ with $[e,f]=h$, $[h,e]=2e$, $[h,f]=-2f$. Then the universal enveloping algebra $U(\mathfrak{sl}_2)$ is isomorphic to the GWA $D(a,\sigma)$, where $D=K[C,H]$, $a=C-H(H+1)$, and $\sigma(H)=H-1$, $\sigma(C)=C$.
\end{proposition}
\begin{proof}
    Set $h'=(1/2) h$. Then $U(\mathfrak{sl}_2)$ is the algebra with generators $h', e, f$, and relations $[e,f]=2h'$, $[h',e]=e$, $[h',f]=-f$. We have a map
    \[
    \phi: K \langle h', e, f \rangle \rightarrow K[C,H](C-H(H+1), \sigma)
    \]
    given by $ h' \mapsto H$, $e \mapsto X^+$, $f \mapsto X^-$, $[H, X^+]= H X^+ - (H-1) X^+= X^+$; and similarly, $[H, X^-]=-X^-$, $[X^+, X^-]= C-(H-1)H - (C -H(H+1))=H(H+1)-H(H-1)=2H$. 
    
    Hence we have a map $\widehat{\phi}: U(\mathfrak{sl}_2) \rightarrow K[C,H](C-H(H+1)), \sigma)$ that is surjective since the Casimir element $\Omega=fe+\displaystyle\frac{1}{4}h(h+2)$ of $U(\mathfrak{sl}_2)$ is mapped onto $C$. Let us now show that $\widehat{\phi}$ is injective. But $\operatorname{GKdim} K[C,H](C-H(H+1)), \sigma) $ is at least $3$ because the base ring $K[C,H]$ has Gelfand-Kirillov dimension 2 according to \cite[Lemma 3.1]{zmz}. If $\widehat{\phi}$ had a non-zero ideal $I$ as kernel, we would have 
    \[
    U(\mathfrak{sl}_2) / I \simeq K[C,H](C-H(H+1)), \sigma).
    \]
On the other hand, by \cite[8.3.2(v), 8.3.6(i)]{mcconnell}, $\operatorname{GKdim} \, U(\mathfrak{sl}_2)/I \leq 2$, as every ideal in a domain is essential. This is impossible. Hence $\operatorname{ker} \widehat{\phi}=0$ and we have our desired isomorphism.
\end{proof}

In more concrete terms, we have a $\mathbb{Z}$-grading of $U(\mathfrak{sl}_2)$ as follows: $U(\mathfrak{sl}_2)=\bigoplus_{n \in \mathbb{Z}} K[\Omega, h] z^n$, where $z^n=e^n$ if $n \geq 0$ and $z^n=f^{-n}$ if $n<0$. The subalgebra corresponding to the degree $0$ is the polynomial algebra $K[\Omega, h]$ generated by $h$ and the Casimir element $\Omega$. Incidentally, we note that it is a maximal commutative subalgebra of $U(\mathfrak{sl}_2)$ with respect to inclusion, the so called \textit{Gelfand--Tsetlin subalgebra} \cite{zhelobenko}.

\begin{theorem}
\label{u(L)isnotpi}
    Let $\mathfrak{g}$ be a Lie algebra over a field $K$ of characteristic $0$. Then $U(\mathfrak{g})$ is a (ungraded) PI-algebra if and only if $\mathfrak{g}$ is abelian.
\end{theorem}
\begin{proof}
See for example, \cite[Theorem 6.50, p. 228]{bahturinbook}. Recall that in the latter reference the case of positive characteristic was also dealt with, in Theorem 6.52. 
\end{proof}

\begin{theorem}
\label{mainthmsl2}
    The ideal of $\mathbb{Z}$-graded polynomial identities of $A=U(\mathfrak{sl}_2)$ is generated by the single identity $[x_1^0,x_2^0]=0$.
\end{theorem}
\begin{proof}
    We clearly have $[x_1^0,x_2^0] \in T_{\mathbb{Z}}(A)$. Consider any nonzero scalar $\lambda \in K^\times$. Then the algebra $A(\lambda)=A/(\Omega-\lambda)$ is a degree 1 generalized Weyl algebra of classical type, $A(\lambda)\simeq K[H](a=\lambda-H(H+1), \sigma)$, $\sigma(H)=H-1$, see \cite[Example 3]{bavula}. Since the quotient map respects the $\mathbb{Z}$-grading, clearly $T_{\mathbb{Z}}(A) \subset T_{\mathbb{Z}}(A(\lambda))$ for every $\lambda \in K^\times$. On the other hand, by Corollary \ref{identities-classical-gwa}, $T_{\mathbb{Z}}(A(\lambda))$ is generated by the single identity $[x_1^0,x_2^0]=0$. Hence $T_{\mathbb{Z}}(U(\mathfrak{sl}_2))$ is generated by the single identity $[x_1^0,x_2^0]=0$.
\end{proof}

We observe that Theorem \ref{mainthmsl2} can be used to give an independent proof of Theorem \ref{u(L)isnotpi} when $\mathfrak{g}=\mathfrak{sl}_2$ in characteristic 0. Observe that $T_{\mathbb{Z}}(U(\mathfrak{sl}_2))=T_{\mathbb{Z}}(W_1)$ implies $T(U\mathfrak{sl}_2)=T(W_1)$ and the latter T-ideal equals 0.

\end{document}